\documentclass[11pt]{amsart}
\usepackage{amssymb,amsmath,amsfonts,amscd,euscript}

\newcommand{\nc}{\newcommand}

\numberwithin{equation}{section}
\newtheorem{thm}{Theorem}[section]
\newtheorem{prop}[thm]{Proposition}
\newtheorem{lem}[thm]{Lemma}
\newtheorem{cor}[thm]{Corollary}
\theoremstyle{remark}
\newtheorem{rem}[thm]{Remark}
\newtheorem{nota}[thm]{Notation}

\newtheorem{dfn}[thm]{Definition}

\nc{\gl}{\mathfrak{gl}}
\nc{\GL}{\mathfrak{GL}}
\nc{\g}{\mathfrak{g}}
\nc{\frk}{\mathfrak{k}}
\nc{\gh}{\widehat\g}
\nc{\h}{\mathfrak{h}}
\nc{\la}{\lambda}
\nc{\C}{\mathbb C }
\nc{\Z}{\mathbb Z }
\nc{\N}{\mathbb N }
\nc{\R}{\mathbb R }
\nc{\Q}{\mathbb Q }
\nc{\al}{\alpha }

\nc{\ta}{\theta}
\nc{\ve}{\varepsilon}
\nc{\ch}{{\mathop {\rm ch}}}
\nc{\Tr}{{\mathop {\rm Tr}\,}}
\nc{\Id}{{\mathop {\rm Id}}}
\nc{\ad}{{\mathop {\rm ad}}}
\nc{\bra}{\langle}
\nc{\ket}{\rangle}
\nc{\x}{{\bf x}}
\nc{\pa}{\partial}
\nc{\ld}{\ldots}
\nc{\cd}{\cdots}
\nc{\hk}{\hookrightarrow}
\nc{\T}{\otimes}
\newcommand{\bea}{\begin{equation}}
\newcommand{\ena}{\end{equation}}
\newcommand{\be}{\begin{equation*}}
\newcommand{\en}{\end{equation*}}
\nc{\gr}{\mathrm{gr}}
\nc{\ov}{\overline}
\newcommand{\bin}[2]{\genfrac{(}{)}{0pt}{}{#1}{#2}}

\nc{\cO}{\mathcal O}
\nc{\msl}{\mathfrak{sl}}
\nc{\mgl}{\mathfrak{gl}}
\nc{\U}{\mathrm U}
\nc{\V}{\EuScript V}

\newcommand{\bc}{{\mathbb C}}
\newcommand{\bz}{{\mathbb Z}}
\newcommand{\bn}{{\mathbb N}}
\newcommand{\fp}{{\mathfrak p}}
\newcommand{\fg}{{\mathfrak g}}
\newcommand{\fb}{{\mathfrak b}}
\newcommand{\ft}{{\mathfrak t}}
\newcommand{\fm}{{\mathfrak m}}
\newcommand{\fn}{{\mathfrak n}}
\newcommand{\fl}{{\mathfrak l}}

\begin{document}

\title[Zhu's algebras, $C_2$-algebras and abelian radicals]
{Zhu's algebras, $C_2$-algebras and abelian radicals}

\author{Boris Feigin, Evgeny Feigin and Peter Littelmann}
\address{Boris Feigin: \newline
Landau Institute for Theoretical Physics,\newline
prosp. Akademika Semenova, 1a, 142432, Chernogolovka, Russia,\newline
{\it and}\newline
Higher School of Economics,\newline
Myasnitskaya ul., 20,  101000,  Moscow, Russia,\newline
{\it and}\newline
Independent University of Moscow,\newline
Bol'shoi Vlas'evski per., 11, 119002, Moscow, Russia
}
\email{bfeigin@gmail.com}
\address{Evgeny Feigin:\newline
Tamm Theory Division,
Lebedev Physics Institute,\newline
Leninisky prospect, 53,
119991, Moscow, Russia,\newline
{\it and }\newline
French-Russian Poncelet Laboratory, Independent University of Moscow
}
\email{evgfeig@gmail.com}
\address{Peter Littelmann:\newline
Mathematisches Institut, Universit\"at zu K\"oln,\newline
Weyertal 86-90, D-50931 K\"oln,Germany
}
\email{littelma@math.uni-koeln.de}

\begin{abstract}
This paper consists of three parts. In the first part we prove that
Zhu's and $C_2$-algebras in type $A$ have the same dimensions.
In the second part we compute the graded decomposition of the 
$C_2$-algebras in type $A$, thus proving the Gaberdiel-Gannon's conjecture.
Our main tool is the theory of abelian radicals, which we develop
in the third part.
\end{abstract}

\maketitle

\section*{Introduction}
The theory of vertex operator algebras plays a crucial role in the
mathematical description of the structures arising from the conformal
field theories. In particular, the representation theory of VOAs
allows one to study correlation functions, partition functions and
fusion coefficients of various theories (see \cite{BF}, \cite{LL}, \cite{K2}).
The Zhu's and $C_2$-algebras
are powerful tools in the representation theory of vertex operator
algebras.

Let $\V$ be a vertex operator algebra. In \cite{Z} Zhu introduced
certain associative algebra  $A(\V)$ attached to $\V$, which turned out
to be very important for the representation
theory of $\V$. In particular, the rationality of $\V$ is conjectured
to be equivalent to finite-dimensionality and semi-simplicity of
$A(\V)$. Moreover, for rational VOAs  the representations of
the vertex algebra and of the Zhu's algebra are in one-to-one correspondence.

There is another algebra (also introduced in \cite{Z}) one can attach to $\V$.
This commutative (Poisson) algebra is called the $C_2$-algebra
and is denoted by $A_{[2]}(\V)$ (we follow the notations from \cite{GG}, \cite{GabGod}).
The $C_2$-algebra as a vector space is the quotient of $\V$ by its $C_2$-subspace.
The finite-dimensionality condition of $A_{[2]}(\V)$ (which is called the 
$C_2$-cofiniteness
condition) is  a very important characterization of $\V$ 
(\cite{ABD}, \cite{CF}, \cite{DM}, \cite{M1},\cite{M2}). In particular, it
implies that the number of isomorphism classes of irreducible $\V$-modules is finite and their characters
have certain modular properties. We note however that finite-dimensionality
of the $C_2$-algebra does not mean that $\V$ is rational (see \cite{GK}).

The algebras $A(\V)$ and $A_{[2]}(\V)$ are very closely related. It has been shown in many
examples that the Zhu's and $C_2$-algebras have the same dimensions and the former can
be regarded as a non-commutative deformation of the latter. But there are also
examples where the equality of dimensions does not hold. Note that
these two algebras can be included into the larger family of spaces and thus be
defined in a uniform way (see \cite{GabGod}, \cite{GG}). One of the special features
of the $A_{[2]}(\V)$ (compared with Zhu's algebra) is that the $C_2$-algebra
comes equipped with the special grading, induced by the conformal weights
(energy) grading on $\V$. Thus $A_{[2]}(\V)=\bigoplus_{m\ge 0} A^m_{[2]}(\V)$.
The spaces $A^m_{[2]}(\V)$ play the central role in this paper.

In this paper we are only interested in the case of vertex-operator algebras, associated
with affine Kac-Moody algebras on the integer level (see \cite{K2}, \cite{BF}).
So let $\g$ be a simple Lie algebra, $\gh$ be the corresponding affine Kac-Moody Lie algebra
(see \cite{K1}). For any non-negative integer $k$ (which is called the level) we denote by
$\V(\g;k)$ the vertex operator algebra associated with $\gh$ on the level $k$.
In particular, as a vector space, $\V(\g;k)$ is isomorphic to the level $k$ basic
(vacuum) representation $L_k$ of $\gh$. It is proved in \cite{FZ} that the corresponding
Zhu's algebra is isomorphic to the quotient of the universal enveloping algebra
$\U(\g)$ by the two-sided ideal generated by the power of the highest root
$e_\theta$:
\[
A(\V(\g;k))=\U(\g)/\bra e_\theta^{k+1} \ket.
\]
The $C_2$-algebra $A_{[2]}(\V(\g;k))$ is the quotient of $L_k$ by the subspace,
spanned by the vectors $(a\T t^{-n})v$, where $a\in \g$, $n\ge 2$ and $v\in L_k$.
In particular, each space $A_{[2]}^m(\V(\g;k))$ is equipped with the structure
of a $\g$-module.
One can see from the definition that $A_{[2]}(\V(\g;k))$
is the quotient of the symmetric algebra $S(\g)\simeq S(\g\T t^{-1})$ by some ideal.
Therefore, it is a very natural question to ask whether the dimensions of Zhu's and
$C_2$-algebras do coincide.
The answer is believed to be positive in many cases (though to the best of our knowledge is not proved
outside the level $1$ case). Note however that for $\g$ of the type $E_8$
the answer is negative (see \cite{GG}).

We show that if $\g$ is of type $A$, then
the $C_2$-algebra is a degeneration of the Zhu's algebra exactly like
$S(\g)$ is the degeneration of $\U(\g)$.
We also prove the conjecture from \cite{GG}, computing the
structure of each $A_{[2]}^m(\V(\g;k))$ as $\g$-module (see also \cite{F}
for the affine version in type $A_1$).

Our approach is based on the following observation: let $\omega_n$
be the "middle" (in the standard Bourbaki numeration \cite{B}) fundamental weight
of the Lie algebra $\gl_{2n}$.
The representation $V(k\omega_n)$ is equipped with the structure of a $\msl_n$-module
via the embeddings
$$
\msl_n\hk \gl_n\hk \gl_n\oplus \gl_n\hk \gl_{2n}
$$
(the first embedding is trivial, the second is the diagonal one and the last
embedding comes from the identification of $\gl_n\oplus\gl_n$ with
the Levi subalgebra of the parabolic subalgebra $\fp_n\subset \gl_{2n}$ corresponding to $\omega_n$).
We construct an isomorphism of $\msl_n$-modules (not of algebras)
$$
V(k\omega_n)\simeq \bigoplus_{i=0}^k A_{[2]}(\V(\msl_n;i),
$$

We derive the graded decomposition of the $C_2$-algebras by
describing $V(k\omega_i)$ as a representation of enveloping algebra of 
nilpotent radical of $\fp$, which is isomorphic to the symmetric algebra  $S(\gl_n)$.

The paper is organized as follows:

In Section $1$ we fix the main definitions and prove that the dimensions of Zhu's and
the $C_2$-algebras coincide in type $A$. The main tool is a deformation argument. 

Knowing the dimension formula, in Section $2$ we compute the graded decomposition of 
$A_{[2]}(\V(\msl_n;k))$ as $\msl_n$-module.

In Section $3$ we discuss some properties of abelian radicals and spherical representations
used in the sections before. This discussion is independent of the type of the algebra and should
be helpful to generalize the arguments to other groups of classical types. 
We work out explicitly the type $A$ case.

\section{Zhu's and $C_2$-algebras in type $A$}
\subsection{Definitions}
Let $\g$ be a simple Lie algebra. Let $\theta$ be the highest root of $\g$
and let $e_\theta\in\g$ be a highest weight
vector in the adjoint representation.
Fix a non-negative integer $k$.
Let $P_k^+(\g)$ be the set of level $k$ integrable $\g$ weights, i.e
the set of dominant integral $\g$ weights $\beta$ satisfying $(\beta,\theta)\le k$.
We denote by $V(\beta)$ the irreducible $\g$-module with highest weight $\beta$.
The following Theorem is proved in  \cite{FZ}:

\begin{thm}
The level $k$ Zhu's algebra $A(\g;k)$ is the quotient
of the universal enveloping algebra $\U(\g)$ by the two-sided
ideal generated by $e_\theta^{k+1}$:
\be
A(\g;k)=\U(\g)/\langle e_\theta^{k+1}\rangle.
\en
In addition, one has the isomorphism of $\g$-modules:
\[
A(\g;k)\simeq\bigoplus_{\beta\in P_k^+(\g)} V(\beta)\T V(\beta)^*.
\]
\end{thm}
The form of the description of $A(\g;k)$ arises ultimately because of the
Peter-Weyl Theorem. 
\begin{nota}
Let $S(\g)=\bigoplus_{m=0}^\infty S^m(\g)$
be the symmetric algebra of $\g$.
For $v\in S^m(\g)$ and $a\in\g$ let $av\in S^{m+1}(\g)$
be the product in the symmetric algebra.
The adjoint action on $\g$ makes 
each homogeneous summand $S^m(\g)$ into 
a $\g$-module. For $v\in S^m(\g)$ and $a\in \g$ we denote by $a\circ v\in S^m(\g)$
the adjoint action of $a$.
\end{nota}

The $C_2$-algebra associated with $\V(\g;k)$ can be described as follows.
The level $k$ $C_2$-algebra $A_{[2]}(\g;k)$ is the quotient
of the symmetric algebra $S(\g)$ by the
ideal generated by the subspace $V_{k+1}=\U(\g)\circ e_\theta^{k+1}\hk S^{k+1}(\g)$:
\be
A_{[2]}(\g;k)=S(\g)/\langle V_{k+1}\rangle.
\en
\begin{rem}
The subspace $V_{k+1}\hk S^{k+1}(\g)$ is isomorphic to the
irreducible $\g$-module $V((k+1)\theta)$ of highest weight $(k+1)\theta$.
The algebra $A_{[2]}(\g;k)$ is naturally a $\g$-module,
the module structure being induced by the adjoint action. Note that 
$A_{[2]}(\g;k)$ is {\bf not} a $\g\oplus\g$-module, differently from $A(\g;k)$.
\end{rem}

Consider the standard filtration $F_\bullet$ on the universal enveloping
algebra $\U(\g)$, such that $\gr_\bullet F\simeq S(\g)$. Let
$F_\bullet(k)$ be the induced filtration on the quotient algebra
$A(\g;k)$.
We have an obvious surjection
\begin{equation}
A_{[2]}(\g;k)\to \gr_\bullet F(k).
\end{equation}
Therefore, we have a surjective homomorphism of $\g$-modules
\begin{equation}\label{ge}
A_{[2]}(\g;k)\to A(\g;k)
\end{equation}
and thus
$\dim A_{[2]}(\g;k)\ge \sum_{\beta\in P_k^+(\g)} (\dim V(\beta))^2 $.
A natural question is: when does this inequality turn into an equality?
In this paper we are also interested in the degree grading on $A_{[2]}(\g;k)$ and in the
corresponding graded decomposition into the direct sum of $\g$-modules. Let
\[
S(\g)=\bigoplus_{m\ge 0} S^m(\g)
\]
be the degree decomposition of the symmetric algebra. This decomposition induces
the decomposition of the $C_2$-algebra:
\begin{equation}
A_{[2]}(\g;k)=\bigoplus_{m\ge 0} A_{[2]}^m(\g;k).
\end{equation}
Each space $A_{[2]}^m(\g;k)$ is naturally a representation of $\g$.
Our main question is as follows:
{\it Find the decomposition of $A_{[2]}^m(\g;k)$
into the direct sum of irreducible $\g$-modules.}
The conjectural answer in type $A$ is given in \cite{GG}.
We prove the conjecture in the next sections.

\subsection{Comparing dimensions in type $A$}
In what follows we restrict ourselves to the case $\g=\msl_n$. So
we omit $\g$ in the notation of Zhu's and $C_2$-algebras:
\[
A(k)=A(\msl_n;k),\quad A_{[2]}(k)=A_{[2]}(\msl_n;k),\quad A_{[2]}^m(k)=A_{[2]}^m(\msl_n;k).
\]
It turns out that it is very convenient to consider the algebras $\msl_n$ and $\gl_n$ and
their representations simultaneously. We fix some notation:
let $\omega_1,\dots,\omega_{n-1}$ be the fundamental weights for $\msl_n$ and
let $\omega_n$ be the additional fundamental weight for $\gl_n$.
For a partition $\la=(\la_1\ge \dots\ge \la_n)$ with $\la_i\in \Z$,  $\la_n\ge 0$ we denote
by $V(\la)$ the corresponding $\gl_n$-module of highest weight:
$$
(\la_1-\la_2)\omega_1+\ldots+(\la_{n-1}-\la_{n})\omega_{n-1}+\la_{n}\omega_{n}.
$$
(We note that $V(\la)$ can be obtained by applying the Schur functor $S_\la$ to the standard
$\gl_n$-module $V$, see for example \cite{FH}, \S6). 
For a partition
$\beta=(\beta_1\ge\dots\ge \beta_{n-1})$, $\beta_i\in\Z$, $\beta_{n-1}\ge 0$ let $V(\beta)$
be the corresponding $\msl_n$-module of highest weight.
$$
(\beta_1-\beta_2) \omega_1 +\dots + (\beta_{n-2}-\beta_{n-1})\omega_{n-2}+\beta_{n-1}\omega_{n-1}.
$$
We have the following restriction isomorphism
\begin{equation}
\left. 
V(\la_1,\dots,\la_n)\right |_{\msl_n}\simeq V(\la_1-\la_n,\dots,\la_{n-1}-\la_n).
\end{equation}
Note also that
\[
P_k^+=P_k^+(\msl_n)=\{(\beta_1,\dots,\beta_{n-1}):\ \beta_1\le k\}.
\]

Our goal is to prove the following theorem:
\begin{thm}\label{theoremone}
$A_{[2]}(k)$ and $A(k)$ are isomorphic as $\msl_n$-modules.
\end{thm}
Because of \eqref{ge}, it is enough to prove that
\begin{equation}
\dim A_{[2]}(k)\le \sum_{\beta\in P_k^+} (\dim V(\beta))^2.
\end{equation}

In what follows we denote by $V$ the $n$-dimensional vector
representation of $\msl_n$. Let $\gl_{n-1}\hk \msl_n$ be the standard
embedding, i.e. $A\in  \gl_{n-1}$ is mapped onto the matrix
$$
\left(\begin{array}{cc}A & 0 \\0 & -a\end{array}\right)\in \msl_n
$$
where $a=\Tr(A)$ is the trace of $A$.
We denote by $U$ the standard $(n-1)$-dimensional vector representation of $\gl_{n-1}$.
Then
\begin{equation}
\msl_n\simeq U\oplus \gl_{n-1}\oplus U^*.
\end{equation}
To be precise, the actions on $U$ and $U^*$ are twisted by one-dimensional representations
given by the characters $\hbox{Tr}$ and $-\hbox{Tr}$. But since these actions are not important 
for the subsequent dimension counting, we omit to mention them explicitly. Note also that
\[
\gl_{n-1}\simeq \gl(U)\simeq U\T U^*, \quad V\simeq U\oplus \C.
\]

We introduce an ``intermediate" algebra $B(k)$.
In the following we often identify $S(\msl_n)$ with the symmetric algebra
$S(U\oplus \gl_{n-1}\oplus U^*)$. Since
\[
U\T V^*\simeq U\oplus U\T U^*\simeq U\oplus \gl(U),
\]
we have embeddings
\begin{equation}
\imath: S^{k+1}(U)\T S^{k+1}(V^*)\hk S^{k+1}(U\oplus \gl(U))\hk S^{k+1}(U\oplus \gl(U)\oplus U^*).
\end{equation}
Similarly,
we have embeddings
\begin{equation}
\imath^*: S^{k+1}(U^*)\T S^{k+1}(V)\hk S^{k+1}(\gl(U)\oplus U^*)\hk S^{k+1}(U\oplus \gl(U)\oplus U^*).
\end{equation}

\begin{dfn}\label{B(k)}
Define the algebra $B(k)$ as the following quotient:
$$
B(k)=S(U\oplus \gl(U) \oplus U^*)/\bra
\imath(S^{k+1}(U)\T S^{k+1}(V^*)), \imath^*(S^{k+1}(V)\T S^{k+1}(U^*))\ket.
$$
\end{dfn}

\begin{rem}\label{remarkdeform}
Let $J=\langle V_{k+1}\rangle\subset S(\g)$ 
be the ideal defining $A_{[2]}(k)$. We are going to define a new grading
on $S(\msl_n)$. Combining the results below, one can view the algebra $B(k)$ as the quotient
of the symmetric algebra $S(\msl_n)$ by the ideal $\tilde J$
obtained from $J$ by the method of the associated cone with respect to this new grading.
The new ideal is stable under the Levi subgroup
\begin{equation}\label{glembedd}
L=GL_{n-1}\subset SL_n,\quad g\mapsto \left(\begin{array}{cc}g & 0 \\0 & (\det g)^{-1}\end{array}\right)
\end{equation}
so $B(k)$ admits an $L$-action. The advantage of the new ideal is the following:
\end{rem}
\begin{lem}\label{frkrealisation}
$B(k)$ admits a $GL_{n-1}\times GL_{n-1}$-action which, restricted to the 
diagonally embedded $L$, is isomorphic to the action of $L$ on $B(k)$ induced
by the adjoint action of $L$ on $\msl_n$. 
\end{lem}
\begin{proof}
The canonical linear map $\msl_n\rightarrow \frk=\{A=(a_{i,j})\in M_n\mid a_{n,n}=0\}$ given by
$$
\left(\begin{array}{cc}A & B \\C & -\hbox{Tr}(A)\end{array}\right) \mapsto \left(\begin{array}{cc}A & B \\C & 0 \end{array}\right)
$$
is an isomorphism of $L$-representations. Now $\frk$
admits an obvious $L\times L$-action: $(\ell_1,\ell_2)\circ A:=\ell_1 A \ell_2^{-1}$, which, restricted
to the diagonal $L$ gives back the adjoint action. The $L\times L$-action also respects the decompositions
$\frk=U\oplus \gl(U)\oplus U^*=U\oplus V\otimes U^*=U\otimes V^* \oplus U^*$ and hence
$$
B(k)\simeq S(\frk) /\langle S^{k+1}(U)\T S^{k+1}(V^*), S^{k+1}(V)\T S^{k+1}(U^*)\rangle
$$
admits an $L\times L$ action with the desired properties.
\end{proof}
\begin{lem}\label{A(k)B(k)}
$\dim A_{[2]}(k)\le \dim B(k)$.
\end{lem}
\begin{proof}
We prove our lemma by introducing a certain grading on $A_{[2]}(k)$ and
keeping only the highest degree terms of the relations. We then show that
the resulting space of relations contains the relations of $B(k)$.

We have on $S(\msl_n)=\oplus_{m=0}^\infty S^m(\msl_n)$ 
the standard degree grading. We introduce
a new grading called the $UU^*$-grading, by setting
\begin{equation}
\msl_n=(\msl_n)_0\oplus (\msl_n)_1,\quad (\msl_n)_0=U\oplus U^*,\ (\msl_n)_1=\gl_{n-1}.
\end{equation}
This decomposition into degree zero and degree one elements induces the $UU^*$-grading on the 
symmetric algebra
$$
S(\msl_n)=\bigoplus_{l\ge 0} S_l(\msl_n).
$$
For an element $f\in  S(\msl_n)$ let $f=\sum_{j=0}^p f_j$, $f_p\not=0$ be a decomposition of $f$ into its
homogeneous parts (with respect to the $UU^*$-grading), and denote by $gr_{UU^*}f:=f_p$
its homogeneous part of highest degree.

For a subspace $I\hk S(\msl_n)$ let $\tilde I$ be the subspace spanned by the highest degree
parts of the elements in $I$:
\[
\tilde I=\langle gr_{UU^*}f\mid f\in I\rangle.
\]
If $I$ is an ideal, then so is $\tilde I$. Moreover,
if $I$ is homogeneous with respect to the standard degree grading, then 
so is $\tilde I$, and for any $m\ge 0$ we have:
\begin{equation}\label{gr}
\dim (S(\msl_n)/I)^m = \dim (S(U\oplus \gl(U)\oplus U^*)/\tilde I)^m.
\end{equation}
The upper index $m$ is used to denote the (standard) degree grading of the quotient space.
We are going to apply the procedure above to the ideal of relations of
the algebra $A_{[2]}(k)$.

Recall that $A_{[2]}(k)=S(\msl_n)/\bra V_{k+1}\ket$.
Because of \eqref{gr}, Lemma \ref{A(k)B(k)} follows from the statement  that
the generators
\[
\imath(S^{k+1}(U)\T S^{k+1}(V^*)), \quad \imath^*(S^{k+1}(V)\T S^{k+1}(U^*))
\]
of the ideal of relations of $B(k)$ are contained in $\widetilde V_{k+1}$.
Let us show that
\begin{equation}\label{tildeI}
\imath(S^{k+1}(U)\T S^{k+1}(V^*))\hk \widetilde V_{k+1}
\end{equation}
(the arguments for the second subspace of relations of $B(k)$ are similar,
see Remark \ref{i*} for more details).
Let
\[
\phi_j: S^{k+1}(U)\T S^j(U^*)\hk S^{k+1}(U)\T S^{k+1}(V^*),\quad j=0,\dots, k+1
\]
be the embedding coming from the decomposition
\[
S^{k+1}(V^*)\simeq S^{k+1}(U^*\oplus \C)=\bigoplus_{j=0}^{k+1} S^j(U^*).
\]
Let us show that
\begin{equation}\label{phij}
\imath\phi_j(S^{k+1}(U)\T S^j(U^*))\hk (\widetilde V_{k+1})_j
\end{equation}
(the $j$-th $UU^*$-degree part of $\widetilde V_{k+1}$). We start with $j=0$. Note that
since $e_\theta\in U$, the power $e_\theta^{k+1}$ belongs to $(\widetilde V_{k+1})_0$.
Since the $UU^*$-grading on $S(\msl_n)$ is $\gl_{n-1}$ invariant, we obtain
\[
\imath\phi_0 (S^{k+1}(U))\hk (\widetilde V_{k+1})_0.
\]
Note: by definition, $e_\theta^{k+1}\in V_{k+1}$ and hence $\imath\phi_0 (S^{k+1}(U))\hk V_{k+1}$.
So in the following we just write $S^{k+1}(U))\hk V_{k+1}$.
Now let us prove \eqref{phij} for $j=1,\dots,k+1$. Let $u_1^*,\dots,u_j^*$ be some elements
of $U^*\hk\msl_n$. Then for any $u_1\dots u_{k+1}\in S^{k+1}(U)\subseteq  V_{k+1}$ 
we have
\begin{equation}\label{uu*}
[u_j^*,[u_{j-1}^*,\dots, [u_1^*, u_1\dots u_{k+1}]\dots]\in V_{k+1}.
\end{equation}
The highest $UU^*$-degree term of \eqref{uu*} is equal to
\begin{equation}\label{hdtuu*}
\sum_{1\le i_1<\dots < i_j\le k+1} u_1\dots u_{i_1-1} [u_1^*,u_{i_1}]\dots [u_j^*, u_{i_j}] u_{i_j+1}\dots u_{k+1},
\end{equation}
because $[U^*,U]\subseteq \gl_{n-1}=(\msl_n)_1$ and $[U^*,\gl_{n-1}]=U^*\hk (\msl_n)_0$.
It is easy to see that the linear span of the elements \eqref{hdtuu*}
(for all $u_i\in U$ and $u_i^*\in U^*$) coincides with
$\imath\phi_j(S^{k+1}(U)\T S^j(U^*))$. Therefore
\[
\imath\phi_j(S^{k+1}(U)\T S^j(U^*))\hk (\widetilde V_{k+1})_j
\]
which finishes the proof of the lemma.
\end{proof}

\begin{rem}\label{i*}
In order to prove the inclusion
\begin{equation}
\imath^*(S^{k+1}(U^*)\T S^{k+1}(V))\hk \widetilde V_{k+1}
\end{equation}
one starts with the element $f_\theta^{k+1}$ ($f_\theta\in\msl_n$
is the lowest weight vector) and introduce the embeddings 
\[
\phi^*_j: S^{k+1}(U^*)\T S^j(U)\hk S^{k+1}(U^*)\T S^{k+1}(V),\quad j=0,\dots, k+1.
\]
Then $\widetilde V_{k+1}=\bigoplus_{j=0}^{k+1} (\widetilde V_{k+1})_j$, where
\[
(\widetilde V_{k+1})_j=\imath\phi_j(S^{k+1}(U)\T S^j(U^*))\oplus \imath^*\phi^*_j(S^{k+1}(U^*)\T S^j(U)),
\quad j=0,\dots,k 
\] 
and
\[
(\widetilde V_{k+1})_{k+1}=\imath\phi_{k+1}(S^{k+1}(U)\T S^{k+1}(U^*))=\imath^*\phi^*_{k+1}(S^{k+1}(U^*)\T S^{k+1}(U)).
\]
\end{rem}
\vskip 3pt
In what follows we often consider quotient algebras of the type
\[
S(W_1\T W_2)/ \bra S^{k+1}(W_1)\T S^{k+1}(W_2)\ket.
\]
In Section \ref{general} (see subsection \ref{typeA}) we describe the $\gl(W_1)\oplus \gl(W_2)$-module 
structure of this algebra explicitly, see also \cite{DEP}. Below we give the description in three special 
cases: $W_1=W_2=U$; $W_1=U$, $W_2=V^*$ and $W_1=V$, $W_2=U^*$.

For a partition $\la$ let $U(\la)$ (resp. $V(\la)$) be the
irreducible $\gl_{n-1}$- (resp. $\gl_{n}$-) module with highest weight $\la$.
Then
\begin{gather}
S(U\T U^*)/\bra S^{k+1}(U)\T S^{k+1}(U^*)\ket\simeq
\bigoplus_{\substack{\la=(\la_1\ge\dots\ge \la_{n-1})\\ k\ge \la_1,\ \la_{n-1}\ge 0}}
U(\la)\T U(\la)^*,\\
\label{UV*}
S(U\T V^*)/\bra S^{k+1}(U)\T S^{k+1}(V^*)\ket\simeq
\bigoplus_{\substack{\la=(\la_1\ge\dots\ge \la_{n-1})\\ k\ge \la_1,\ \la_{n-1}\ge 0}}
U(\la)\T V(\la)^*,\\
\label{U*V}
S(V\T U^*)/\bra S^{k+1}(V)\T S^{k+1}(U^*)\ket\simeq
\bigoplus_{\substack{\la=(\la_1\ge\dots\ge \la_{n-1})\\ k\ge \la_1,\ \la_{n-1}\ge 0}}
V(\la)\T U(\la)^*.
\end{gather}

\begin{prop}
$\dim B(k)\le \sum_{\beta\in P_k^+} (\dim V(\beta))^2.$
\end{prop}
\begin{proof}
We regard $B(k)$ as a $L\times L$-module, recall that $\msl_{n}$ decomposes
as 
$$
\msl_{n}=U\otimes U^* \oplus U\otimes \bc_{\det}\oplus \bc_{\det^{-1}}\otimes U^*.
$$
Here $\bc_{\det}$ and $\bc_{\det^{-1}}$ denote the one-dimensional representations
given by the determinant respectively its inverse. In the following we omit to indicate
the twists by these characters since they are not relevant for the dimension counting.
Let $B_0(k)\hk B(k)$ be the subalgebra generated by $\gl_{n-1}\simeq U\T U^*$. 
Since the subspace
\[
S^{k+1}(U)\T S^{k+1}(U^*)\hk S^{k+1}(\gl_{n-1})
\]
is sitting inside the space of relations of $B(k)$, there exists  an embedding of
$L\times L$-modules
\begin{equation}\label{B0}
B_0(k)\hk \bigoplus_{\substack{\la=(\la_1\ge\dots\ge \la_{n-1})\\ k\ge \la_1,\ \la_{n-1}\ge 0}}
U(\la)\T U(\la)^*.
\end{equation}
We have a surjective homomorphism of $L\times L$-modules
\begin{equation}
S(U)\T B_0(k)\T S(U^*)\to B(k).
\end{equation}
By \eqref{B0}, we obtain the $L\times L$-equivariant surjection
\begin{equation}\label{-0+}
\bigoplus_{\substack{\la=(\la_1\ge\dots\ge \la_{n-1})\\ k\ge \la_1,\ \la_{n-1}\ge 0}}
S(U) \T \left( U(\la)\T U(\la)^*\right) \T  S(U^*)\to B(k).
\end{equation}
In the following we write $\mu\succeq\la$ for {\it $\mu$ is larger and in between $\la$}, i.e.:
$$
\mu\succeq\la\Leftrightarrow \mu_1\ge\la_1\ge\mu_2\ge\la_2\ge  \ldots\ge \mu_{n-1}\ge\la_{n-1},
$$ 
and we write $\vert\la\vert$ for the sum $\sum_i \la_i$.
By the Pieri formula (see for example \cite{FH}, Appendix A) the left hand side of \eqref{-0+} is,
as $L\times L$-module, isomorphic to the direct sum
\begin{equation}\label{munu}
\bigoplus_{\substack{\la=(\la_1\ge\dots\ge \la_{n-1})\\ k\ge \la_1,\ \la_{n-1}\ge 0}}
\bigoplus_{\mu,\nu\succeq \la}
U(\mu)\T U(\nu)^*.
\end{equation}
Because of the relations in $B(k)$, not all the tensor products as in \eqref{munu}
do really appear in the decomposition of $B(k)$ as $L \times L$-module.
In fact, \eqref{UV*} and \eqref{U*V} imply that if  $U(\mu)\T U(\nu)^*$ appears in
$B(k)$ (up to twists by a character), then $\mu$ and $\nu$ are restricted by $\mu_1\le k$ and $\nu_1\le k$.
Thus our proposition follows from the equality
\begin{equation}\label{dual}
\sum_{\beta\in P_k^+} (\dim V(\beta))^2=
\sum_{\la:\ k\ge \la_1}
\sum_{\substack{k\ge \mu_1,\nu_1\\ \mu,\nu\succeq \la}}
\dim U(\mu) \dim U(\nu)^*,
\end{equation}
which is going to be proved in the following lemma.
\end{proof}

\begin{lem}
The equality in \eqref{dual} holds.
\end{lem}
\begin{proof}
We decompose each $V(\beta)$ and $V(\beta)^*$ into the direct sum of
$\gl_{n-1}$-modules following the Gelfand-Tseitlin rule (see \cite{GT}) and keep in mind
the embedding of $L=GL_{n-1}$ in $SL_n$ (see (\ref{glembedd})). Thus our
lemma is equivalent to the following equation (again we omit twists by characters):
\begin{multline}
\sum_{\substack{\beta=(\beta_1\ge\dots\ge \beta_{n-1})\\ k\ge \beta_1,\beta_{n-1}\ge 0}}
\sum_{\mu,\nu \preceq \beta}
\dim U(\mu)\dim U(\nu)^*\\ 
=
\sum_{\substack{\la=(\la_1\ge\dots\ge \la_{n-1})\\ k\ge \la_1,\ \la_{n-1}\ge 0}}
\sum_{\substack{\bar\mu, \bar\nu\succeq \la\\  \bar\mu_1,\bar\nu_1\le k }}
\dim U(\bar\mu) \dim U(\bar\nu)^*.
\end{multline}
To prove the latter we use the bijection between the parameter sets
\begin{gather*}
(\beta,\mu,\nu)\mapsto (\la,\bar\mu,\bar\nu),\\
\la_i=k-\beta_{n-i-1},\ \bar\mu_i=k-\mu_{n-1-i},\ \bar\nu_i=k-\nu_{n-1-i}.
\end{gather*}
\end{proof}

\begin{cor}\label{AA}
$A(k)$ and $A_{[2]}(k)$ are isomorphic as $\msl_n$-modules.
\end{cor}
\begin{proof}
We have
\begin{multline*}
\sum_{\beta\in P_k^+} (\dim V(\beta))^2=\dim A(k)\le \dim A_{[2]}(k)\\
\le \dim B(k)\le \sum_{\beta\in P_k^+} (\dim V(\beta))^2.
\end{multline*}
Hence $\dim A(k)=\dim A_{[2]}(k)$, which by \eqref{ge} implies the corollary.
\end{proof}
\begin{proof} ({\sl of Theorem~\ref{theoremone}}) The theorem is a consequence of Corollary \ref{AA}.
\end{proof}
As another consequence we get: the inequality 
$\dim S(\msl_n)/\langle \widetilde V_{k+1}\rangle\le \dim B(k)$ in the proof
of Lemma~\ref{A(k)B(k)} is an equality and hence, as claimed in Remark~\ref{remarkdeform}:
\begin{cor}
$B(k)=S(\msl_n)/\langle \widetilde V_{k+1}\rangle$.
\end{cor}
\begin{cor}\label{AB}
$\dim A_{[2]}^m(k)=\dim B^m(k)$ for all $m\ge 0$.
\end{cor}

\section{The graded decomposition of $A_{[2]}(k)$}
In this section we compute the graded decomposition of $C_2$-algebras, thus
proving the conjecture of \cite{GG}.

\begin{dfn}
We define the algebra $C(k)$ as follows:
\[
C(k)=S(V\T V^*)/\bra S^{k+1}(V)\T S^{k+1}(V^*)\ket
\]
\end{dfn}
\begin{rem}\label{permute}
Let $E_{a,b}\in\gl_n$ be the standard matrix unit. Identifying $V\T V^*$ with $\gl(V)$
we can describe the subspace $S^{k+1}(V)\T S^{k+1}(V^*)$ as the linear span
of the monomials of the form
\begin{gather*}
\sum_{\sigma\in S_{k+1}}  E_{i_1,j_{\sigma_1}} E_{i_2,j_{\sigma_2}}\dots E_{i_{k+1},j_{\sigma_{k+1}}},\\
1\le i_1\le\dots\le  i_{k+1}\le n, 1\le j_1\le\dots\le j_{k+1}\le n,
\end{gather*}
where the sum is taken over the permutation group $S_{k+1}$.
\end{rem}
\begin{rem}
Using the identification in Lemma \ref{frkrealisation}, we see that with respect to the embeddings
$L=GL_{n-1}\subset SL_n\subset GL_n$ we have:
$B(k)\simeq C(k)/\langle E_{n,n} \rangle$, also as a $L\times L$-module.
\end{rem}
The standard degree grading $S^m$  on the symmetric algebra induces a grading
on $C(k)$:
\[
C(k)=\bigoplus_{m\ge 0} C^m(k).
\]
Recall that we have the decomposition with respect to the action of $\gl_n\oplus\gl_n$:
\begin{equation}\label{2n}
C(k)\simeq \bigoplus_{\substack{\la=(\la_1\ge\dots\ge \la_n)\\ k\ge \la_1,\ \la_n\ge 0}}
V(\la)\T V(\la)^*.
\end{equation}
Moreover, using the Cauchy formula (see for example \cite{P}, \S9) for $S^m(V\otimes V^*)$ we see:
\begin{equation}\label{char}
C^m(k)\simeq \bigoplus_{\substack{\la:\ k\ge \la_1,\ \la_n\ge 0\\ \la_1+\dots + \la_n=m}}
V(\la)\T V(\la)^*.
\end{equation}
We will extract the information about $A_{[2]}^m(k)$ from \eqref{char}.

Let $\msl_n\hk \gl_n$ be again the standard embedding. As $\msl_n$-module (adjoint action)
the Lie algebra $\gl_n\simeq V\T V^*$ decomposes as
\begin{equation}\label{b}
\gl_n=\msl_n  \oplus \C c,
\end{equation}
where $\C c$ is
trivial one-dimensional module with fixed non-trivial vector $c$ (say, $c\in V\T V^*$
is the identity operator). The algebra $C(k)$ is by construction a $S(\gl_n)$- as well as a $S(\msl_n)$-module.
For $i\ge 0$ let $D_i\subset C(k)$ be the $S(\msl_n)$-submodule of $C(k)$ generated $c^i$:
\[
D_i=S(\msl_n)\cdot c^i \subset  C(k).
\]
\begin{lem}\label{sub}
The canonical surjective map $S(\msl_n)\to D_0\subset C(k)$, $f\mapsto f\cdot 1$,
induces a surjective homomorphism $A_{[2]}(k)\to D_0$ of 
$\msl_n$-modules.
\end{lem}
\begin{proof} By definition, the image of $S(\msl_n)$ in $C(k)$ is $D_0$.
Now $e_\theta=E_{1,n}$ and by Remark \ref{permute}
$e_\theta^{k+1}\in S^{k+1}(V)\T S^{k+1}(V^*)$, which proves the lemma.
\end{proof}

In general the following proposition holds.
\begin{prop}\label{quot}
For all $i=0,1,\dots,k$ we have a surjective homomorphism of
$\msl_n$-modules
\[
A_{[2]}(k-i)\to \left. \sum_{j=0}^i D_j\right /\sum_{j=0}^{i-1} D_j
\]
and $D_{k+1}\hk\sum_{j=0}^k D_j$.
\end{prop}
\begin{proof} 
It suffices to show that for any $0\le i\le k+1$ the following
is true in the symmetric algebra $S(V\T V^*)$:
\begin{equation}
e_\theta^{k+1-i} c^i\in S^{k+1}(V)\T S^{k+1}(V^*) + \sum_{l=0}^{i-1} D_l.
\end{equation}
Let $E_{a,b}\in\gl_n$ be the matrix with entries $(\delta_{i,a}\delta_{b,j})_{1\le i,j\le n}$.
Note that
\[
e_\theta=E_{1,n},\quad c=\sum_{a=1}^n E_{a,a}.
\]
We first write
\begin{equation}\label{s}
e_\theta^{k+1-i} c^i=\left[(c-nE_{n,n})^i +
\sum_{j=0}^{i-1} n^{i-j}\bin{i}{j}  (c-nE_{n,n})^jE_{n,n}^{i-j} \right] E_{1,n}^{k+1-i}.
\end{equation}
Let us show that for any $j=0,\dots,i-1$
\begin{equation}\label{Enn}
E_{n,n}^{i-j}E_{1,n}^{k+1-i}\in S^{k+1}(V)\T S^{k+1}(V^*) + \sum_{l=0}^{i-1} D_l.
\end{equation}
For the $j=0$ term (because of Remark \ref{permute}) we have
\[
E_{n,n}^i E_{1,n}^{k+1-i}\in S^{k+1}(V)\T S^{k+1}(V^*).
\]
Note that since $E_{1,n}\in \msl_n$ the general $j$ case of \eqref{Enn}
follows from the statement
that for any $p\ge 0$
\begin{equation}\label{p}
E_{n,n}^p \in S(\msl_n)\mathrm{span}\{ 1,c,\dots, c^p\}
\end{equation}
We prove \eqref{p} by induction on $p$. For any $\al=1,\dots,n$ we have
\[
E_{n,n}^p - E_{n,n}^{p-1}E_{\al,\al}= E_{n,n}^{p-1} (E_{n,n}-E_{\al,\al}).
\]
By induction assumption $E_{n,n}^{p-1}\in S(\msl_n)\mathrm{span}\{ 1,c,\dots, c^{p-1}\}$ and
thus summing up over all $\al$
$$
nE_{n,n}^p - c E_{n,n}^{p-1}\in S(\msl_n)\mathrm{span}\{ 1,c,\dots, c^{p-1}\}.
$$
Using the induction assumption again we arrive at
$$
E_{n,n}^p \in S(\msl_n)\mathrm{span}\{ 1,c,\dots, c^p\}.
$$
This finishes the proof of the proposition.
\end{proof}

Recall (see \eqref{sym}) that $C(k)$ is isomorphic to the $\gl_{2n}$-module
$$
V(\underbrace{k,\dots,k}_n,\underbrace{0,\dots,0}_n).
$$ 
We endow $C(k)$ with a structure of $\msl_n$-module
via the chain of embeddings
\[
\msl_n\hk\gl_n\hk\gl_n\oplus\gl_n\hk\gl_{2n}.
\]
Here the first embedding is the standard one, the second is the diagonal embedding and
the last one comes from the isomorphism of $\gl_n\oplus\gl_n$
and the Levi subalgebra of $\gl_{2n}$, corresponding to $\omega_n$
($E_{i,j}\oplus E_{k,l}\mapsto E_{i,j}+E_{n+k,n+l}$).
So in what follows we consider $C(k)$ equipped with this structure of $\msl_n$-module.
\begin{thm}
We have an isomorphism of $\msl_n$-modules
\[
C(k)\simeq \bigoplus_{i=0}^k A_{[2]}(k-i).
\]
Moreover,
\[
C^m(k)\simeq \bigoplus_{i=0}^{\min(m,k)} A_{[2]}^{m-i}(k-i).
\]
\end{thm}
\begin{proof}
Because of Lemma \ref{sub} and Proposition \ref{quot} it suffices to prove that
\[
\sum_{j=0}^k \dim A_{[2]}(j)=\dim C(k).
\]
Because of Corollary \ref{AA} this is equivalent to
\[
\sum_{j=0}^k \dim A(j)=\dim C(k).
\]
Recall that
\begin{multline}\label{lan}
\dim C(k)
=\sum_{\substack{\la=(\la_1\ge\dots\ge \la_n)\\ k\ge \la_1,\ \la_n\ge 0}} (\dim V(\la))^2\\
=\sum_{\la_n=0}^k \sum_{\beta=(\la_1-\la_n,\dots,\la_{n-1}-\la_n)} (\dim V(\beta))^2,
\end{multline}
where $V(\beta)$ in the last line are  irreducible $\msl_n$-modules.
We note that if $\beta_i=\la_i-\la_n$ as above, then the $\msl_n$-module 
$V(\beta)$ and $\gl_n$-module $V(\la)$ are isomorphic as vector spaces.
Note also that if $\la_n$ is fixed, the restriction $k\ge \la_1$ turns into
$\beta_1\le k-\la_n$. Therefore, \eqref{lan} can be rewritten as
\[
\dim C(k)=\sum_{\la_n=0}^k \dim A(k-\la_n),
\]
which proves the theorem.
\end{proof}

\begin{cor}
We have an isomorphism of $\msl_n$-modules
\[
A^m_{[2]}(k)\simeq C^m(k)/C^{m-1}(k-1).
\]
\end{cor}
\begin{cor}\label{GG}
We have an isomorphism of $\msl_n$-modules
\[
A^m_{[2]}(k)=\frac
{\bigoplus_{\substack{\la:\ k\ge \la_1,\ \la_n\ge 0\\ \la_1+\dots + \la_n=m}} V(\la)\T V(\la)^*}
{\bigoplus_{\substack{\la:\ k-1\ge \la_1,\ \la_n\ge 0\\ \la_1+\dots + \la_n=m-1}} V(\la)\T V(\la)^*},
\]
where the $\gl_n$-module $V(\la)$ is regarded as the $\msl_n$-module
with the highest weight $(\la_1-\la_n,\dots,\la_{n-1}-\la_n)$.
\end{cor}

Corollary \ref{GG} is a restatement of the Conjecture of Gaberdiel and Gannon
(see \cite{GG}, formulas $(4.4)$ and $(4.5)$). Recall that, despite the suggesting formula, $A^m_{[2]}(k)$ is 
{\bf not} a $\msl_n\oplus \msl_n$-module.

\section{Abelian radicals and spherical modules}\label{general}
Let $G$ be a simple simply connected algebraic group with Lie algebra $\fg$. We fix a Borel subgroup
$B$ with Lie algebra $\fb$ and a maximal torus $T\subset B$ with Lie algebra $\ft$. Let $\Phi$
be the root system and let $\Phi^+$ be the set of positive roots and $\Delta$ the set of simple roots
corresponding to the choice of $\fb$.

Let $\fp$ be a maximal standard parabolic subalgebra of the simple Lie algebra $\fg$, i.e.
$\fp$ contains the fixed Borel subalgebra $\fb$.
Then $\fp$ is completely determined
by the only simple root $\alpha$ such that the root space $\fg_{-\alpha}$ is not contained in
$\fp$, we call $\-\alpha$ the simple root associated to $\fp$.

Let $\fp=\fl\oplus\fm$ be a Levi decomposition, i.e. $\fl$ is a reductive Lie subalgebra containing $\fl$
and $\fm$ is the nilpotent radical of $\fp$.

Let $P$ be the corresponding maximal parabolic subgroup of $G$ with Levi decomposition $P=LP^u$,
$L\supset T$. Then $L$ and $\fl$ act on $\fm$ via the adjoint action.

The opposite parabolic subalgebra is denoted by $\fp^-$. Let $\fp^-=\fl\oplus \fm^-$ be its Levi decomposition,
then $\fm^-\simeq\fm^*$ as $L$-module.
\subsection{Certain annihilators}
Corresponding to the choice of a maximal parabolic and its associated root $\alpha$ let $\omega$
be the associated fundamental weight. Let $\fn$ be the nilpotent radical of $\fb$ and let $\fn^-$ be the
nilpotent radical of the opposite Borel subalgebra $\fb^-$.

We define a $\bz$-grading on the root system $\Phi$ by setting for the simple roots
$\text{deg\,}\alpha=1$ and $\text{deg\,}\gamma=0$ for $\gamma\not=\alpha$.
Let $\fg=\bigoplus_{j\in\bz}\fg_j$ be the corresponding $\bz$-grading,
then
$$
\fl=\fg_0,\quad \fp=\bigoplus_{j\ge 0}\fg_j\quad\fm=\bigoplus_{j\ge 1}\fg_j\quad \text{and}\quad\fm^-=\bigoplus_{j\le -1}\fg_j
$$
We choose a basis $X_{-\beta}\in\fg_{-\beta}$, $\beta\in\Phi^+$, of root vectors for $\fn^-$.
Fix a highest weight vector $v_{k\omega}$ in $V(k\omega)$. For all $k\in\bn$ we have a surjective map
$$
\Psi:U(\fn^-)\rightarrow V(k\omega),\quad n\mapsto nv_{k\omega},
$$
the kernel being the left ideal in $U(\fn^-)$ generated by
$$
\{X_{-\alpha}^{k+1}, X_{-\gamma}\mid \gamma\in\Delta,\deg\gamma=0\}.
$$
The restriction of this map to $U(\fm^-)$ is already surjective:
$$
\pi_k:U(\fm^-)\rightarrow V(k\omega),\quad m\mapsto mv_{k\omega}.
$$
To describe a generating system for the kernel of $\pi_k$,
recall that $\fm^-$ is a $L$-module via the adjoint action of $\fl$ and $L$ on $\fg$.
We write $\ell\circ m$ for the adjoint action.
Set $\fl':=[\fl,\fl]$.
\begin{lem}\label{generalannihilator}
The kernel of $\pi_k$ is the left $U(\fm^-)$-ideal generated by the irreducible $L$-module
$\langle L\circ X_{-\alpha}^{k+1}\rangle=U(\fl') \circ X_{-\alpha}^{k+1}$, and $\pi_k$
induces an isomorphism of $L$-modules $U(\fm^-)/\ker\pi_k\simeq V(k\omega)\otimes\bc_{-k\omega}$,
where $\bc_{-k\omega}$ denotes the one-dimensional representation associated to the $L$-character $-k\omega$. 
\end{lem}
\proof
If $p_1,\ldots, p_r\in {\fm^-}$ and $X\in {\mathfrak \fl'}$, then
\begin{multline*}
X.p_1\cdots p_r=\sum_{j=1}^r p_1\cdots [X,p_j]\cdots p_r+p_1p_2\cdots p_rX\\
=X\circ (p_1\cdots p_r)+p_1p_2\cdots p_rX.
\end{multline*}
So if $m_1,\ldots,m_q$ are monomials in $U(\fm^-)$ and $\sum_{j=1}^q a_j m_j \in\ker\pi_k$,
then
$$
X\circ(\sum_{j=1}^q a_j m_j )= X.(\sum_{j=1}^q a_j m_j )-(\sum_{j=1}^q a_j m_j ).X
$$
annihilates $v_{k\omega}$ too since $X$ annihilates $v_{k\omega}$. Hence the left ideal
generated by $U(\fl') \circ X_{-\alpha}^{k+1}$ is contained in $\ker\pi_k$.
Set $\fn_0^-=\fn^-\cap\fl$, then
$$
\begin{array}{rcl}
\ker \Psi &=& U(\fn^-)X_{-\alpha}^{k+1}+ \sum_{\gamma\in \Delta,\deg \gamma =0} U(\fn^-)X_{-\gamma}\\
&=&U(\fm^-)U(\fn^-_0)X_{-\alpha}^{k +1}+ \sum_{\gamma\in \Delta,\deg \gamma =0} U(\fn^-)X_{-\gamma}\\
&=&U(\fm^-)\big(U(\fn^-_0)\circ X_{-\alpha}^{k +1}\big)+
\sum_{\gamma\in \Phi^+,\deg\gamma=0} U(\fn^-)X_{-\gamma}.\\
\end{array}
$$
Now $X_{-\alpha}$ (and hence $X_{-\alpha}^{k+1}$ too) is a highest weight vector for the action of
$U(\fl')$ and hence $U(\fl') \circ X_{-\alpha}^{k+1} = U(\fn^-_0)\circ X_{-\alpha}^{k +1}$. So
the image of $\ker\Psi$ in $U(\fn^-)/U(\fn^-).\fn_0^-\simeq U(\fm^-)$ coincides with the
left ideal described in the lemma, which shows the equality.

To make the induced isomorphism $U(\fm^-)/\ker\pi_k\simeq V(k\omega)$ $L$-equivariant,
we need to twist the representation by the appropriate character.
\endproof
\subsection{Abelian and spherical radicals}

The action of a reductive group $H$ on a (possibly infinite dimensional) locally finite representation space $V$  
is called {\it multiplicity free} if the decomposition of the $H$-module  $V$ into the direct sum of irreducible 
finite dimensional $H$-modules contains any irreducible $L$-module with multiplicity at most one. The action
of $H$ on an affine variety $M$ is called {\it multiplicity
free} if $\bc[M]$ is multiplicity free.

Note that the action of $L$ on $\bc[\fm^-]$ is multiplicity free if and only the action of $L$ on
$U(\fm^-)$ is multiplicity free. In particular, in this case the irreducible $L$-module in $U(\fm^-)$
spanned by the generating set for the ideal $\ker\pi_k$ (see Lemma~\ref{generalannihilator})
is uniquely determined by the highest weight $-k\alpha$.

Let $\theta$ denote the highest root of $\fg$. A simple root $\alpha$ is called {\it co-minuscule} if
$\alpha$ occurs with coefficient $1$ in the expression of $\theta$ as a sum of simple roots.
\begin{prop}
{\it i)} The nilpotent radical $\fm^-$ of $\fp^-$ is abelian if and only if the simple root $\alpha$ associated to $\fp$
is co-minuscule. In this case the action of $L$ on $\fm^-$ is irreducible.
(For a  list of co-minuscule weights see Remark~\ref{minu}.)

{\it ii)} The action of $L$ on $\bc[\fm^-]$ is multiplicity free if and only if either $\fm^-$ is
abelian, or $\alpha=\alpha_n$ for $\fg$ of
type ${\tt B}_n$, or  $\alpha=\alpha_1$ for $\fg$ of type ${\tt C}_n$.
\end{prop}
\begin{rem}\label{minu}
\rm
Let $\Delta=\{\alpha_1,\ldots,\alpha_n\}$ be the simple roots, indexed as in
\cite{B}, Planches I-IX. According to the tables in \cite{B} we find the following list of
co-minuscule roots $\alpha_i$, or, equivalently, the list of roots such that the corresponding maximal parabolic
subalgebra has an abelian nilpotent radical: in type ${\tt A}_n$ all roots are co-minuscule, 
in type ${\tt B}_n$ only the root $\alpha_1$ and
in type ${\tt C}_n$ only the root $\alpha_n$ are co-minuscule.
In type ${\tt D}_n$ the roots $\alpha_1,\alpha_{n-1}, \alpha_n$ and
in type ${\tt E}_6$ the roots $\alpha_1,\alpha_{6}$ are co-minuscule. 
In type ${\tt E}_7$ only the root $\alpha_7$ is co-minuscule.
For the types ${\tt E}_8$, ${\tt F}_4$ and ${\tt G}_2$ there are no maximal parabolic subalgebras
with abelian nilpotent radical.
\end{rem}
\proof
The first part {\it i)} is well known:
Let $\Phi_\fm^+$ be the positive roots occurring in $\fm$. If the last condition holds, then for two roots
$\beta,\gamma\in \Phi_\fm^+$ the sum is never a root and hence the commutator of the two root subspaces
is zero. If the condition fails, then we can find two roots $\beta,\gamma\in \Phi_\fm^+$ such that the sum is a root,
and hence the corresponding root subspaces fail to commute. It follows now from elementary properties of root systems
that the $\fl$-action on $\fm$ is irreducible, having the root $\alpha$ as lowest weight vector.

A comparison of the list of co-minuscule roots with the tables in \cite{K3} shows that in these
cases the representations of $L$ on $\fm$ are multiplicity free.

Another approach, which gives the complete list at once is the following: the action
of a reductive group $H$ on a variety $Y$ is called {\it spherical} if a Borel subgroup of $H$ has a dense orbit in $Y$.
It is know for $Y$ affine that the conditions spherical and multiplicity free are equivalent.
The action of $L$ on $\fm$ is spherical if and only if it is so for the action on $\fm^-$. Via the exponential
we see that the action of $L$ on is $\fm^-$ is spherical if and only if the action of $L$ on the unipotent radical $P^{-,u}$
of the opposite parabolic subgroup is so.
Now $P^{-,u}$ can be identified in a $L$-equivariant way with the open cell $P^{-,u}.id\subset G/P$, so we
see that the action of $L$ on $\fm$ is multiplicity free if and only if the action of $L$ on $G/P$ is spherical.
These maximal parabolic subgroups have been classified in \cite{L1}.
\endproof

\subsection{Type $A$ case}\label{typeA}
In this subsection we work out explicitly the type $A$ case of the general constructions
explained in Section \ref{general} above.

So let $\g=\gl_N$ and let
\[
\omega_i=(\underbrace{1,\dots,1}_i), \quad i=1,\dots,N
\]
 be some fundamental weight.
Then $\fl=\gl_i\oplus \gl_{N-i}$ and the abelian radical $\fm^-$ is isomorphic
to $(\C^i)^*\T \C^{N-i}$ as $\fl$-module ($\C^i$ and $\C^{N-i}$ are vector representations of
$\gl_i$ and $\gl_{N-i}$).
In what follows we deal with $\gl_N$, $\gl_i$ and $\gl_{N-i}$-modules simultaneously,
so we use the notation $V_j(\la)$ to denote the irreducible $\gl_j$-module
with highest weight $\la$. 

The $\fl$-module $U(\fl)\circ X_{-\al_i}^{k+1}$ is isomorphic to
\begin{equation}\label{sym}
S^{k+1} (\C^{i*})\T S^{k+1} (\C^{N-i})\hk S^{k+1}(\fm^-).
\end{equation}
We write $(k^i)$ for the partition $(\underbrace{k,\ldots,k}_i) $. 
Then Lemma \ref{generalannihilator} in type $A$ reads as:
\begin{equation}\label{repasalgebra}
V_N((k^i))\otimes \bc_{-k\omega_i}\simeq S(\C^{i*}\T \C^{N-i})/\bra S^{k+1} (\C^{i*})\T S^{k+1} (\C^{N-i}) \ket.
\end{equation}

In the following lemma we describe $V_N((k^i))$ as $\fl=\gl_i\oplus \gl_{N-i}$-module
(the $\fl$ action on $V_N(k\omega_i)$ coming from the adjoint action of
$\fl$ on $\fm^-$).

\begin{lem}\label{branching}
The module $S(\C^{i*}\T \C^{N-i})/\bra S^{k+1} (\C^{i*})\T S^{k+1} (\C^{N-i}) \ket$
(respectively $V_N((k^i))\otimes \bc_{-k\omega_i}$) is
isomorphic (as $\gl_i\oplus \gl_{N-i}$-module)
to the direct sum
\begin{equation}\label{ALevi}
\bigoplus_\la V_i(\la)^*\T V_{N-i}(\la),
\end{equation}
where the sum is running over all partitions $\la=(\la_1\ge\dots\ge \la_{\min(i,N-i)})$
such that $k\ge \la_1$.
\end{lem}
\begin{proof} Using \eqref{repasalgebra}, the lemma can be deduced from \cite{DEP}.
An alternative approach is to use branching rules for Levi subgroups
of $GL_N$. The rule is due to Littlewood, the formulation
used here can be found in \cite{L2}.

We briefly explain how the formula \eqref{ALevi} shows up in type $A$.
The partition $(k^i)$  is represented by the rectangular Young diagram
having $i$-rows and $k$-columns. The restriction formula
for Levi subalgebras can in this case be read as follows: the irreducible components
of $V_N((k^i))$ as $\gl_i\oplus \gl_{N-i}$-module are of the form
$V_i((k^i-\la)) \otimes V_{N-i}(\la)$ where the partition $\la$ is obtained by cutting the
rectangle $(k^i)$ into two partitions: $\la=(\la_1,\ldots,\la_i)$ and
$(k^i-\la)=(k-\la_i,\dots, k-\la_1)$. Here we assume $\la_{N-i+1}=\dots=\la_i=0$ if $i>N-i$.
But note that $V_i((k^i-\la))\simeq V_i(\la)^*\T\bc_{k\omega_i}$, which proves the lemma.
\end{proof}

\section*{Acknowledgments}
The work of EF was partially supported by the Russian President Grant MK-281.2009.1,
by the RFBR Grants 09-01-00058, 07-02-00799 and NSh-3472.2008.2, by Pierre Deligne fund
based on his 2004 Balzan prize in mathematics and by Alexander von
Humboldt Fellowship. The work of BF was partially supported by 
RFBR  initiative interdisciplinary project grant 09-02-12446-ofi-m and by
RFBR-CNRS grant 09-02-93106.
The work of PL was partially supported by the 
priority program SPP 1377 of the German Science Foundation.

\end{document}